\documentclass[12pt,a4paper]{amsart}
\newtheorem{theorem}{Theorem}
\newtheorem*{lemma}{Lemma}
\theoremstyle{definition}
\newtheorem*{definition}{Definition}

\DeclareMathOperator{\Length}{L}

\begin{document}

\title{On Shirshov bases of graded algebras}

\author{Fedor Petrov}
\address{St. Petersburg Department of Steklov Institute of Mathematics
and St. Petersburg State University}
\email{fedyapetrov@gmail.com}

\author{Pasha Zusmanovich}
\address{Department of Mathematics, Tallinn University of Technology}
\email{pasha.zusmanovich@ttu.ee}

\date{last revised September 4, 2012}
\thanks{\textsf{arXiv:1201.5837}; Israel J. Math., to appear}

\begin{abstract}
We prove that if the neutral component in a finitely-generated associative 
algebra graded by a finite group has a Shirshov base, then so does the whole 
algebra.
\end{abstract}

\maketitle

The following result, established by Shirshov in 1957, is one of the 
cornerstones of the modern PI theory.
\begin{theorem}[Shirshov's height theorem]
For any finitely-generated PI algebra $A$, there exist a natural number $h$
and a finite subset $S \subset A$, such that $A$ is linearly spanned
by elements of the form $a_1^{k_1} \dots a_n^{k_n}$, where $n \le h$, 
$k_1, \dots, k_n \in \mathbb N$, and $a_i\in S$.
\end{theorem}

``Algebra'' means an associative algebra defined over a commutative ring with 
unit. As usual, an algebra is called PI, if it satisfies a nontrivial
polynomial identity.

This theorem immediately yields a positive solution of Kurosh's problem
about finite-dimensionality of a finitely-generated algebraic PI algebra,
boundedness of the Gelfand--Kirillov dimension of PI algebras, and many other 
similar questions, and has sparkled a lot of subsequent work 
(see \cite{bbl}, \cite{belov-rowen} and \cite{kemer} for excellent surveys). 
In the majority of these applications, one utilizes solely the condition 
of the conclusion of Shirshov's height theorem, without further appeal
to PIness of algebra in question. Therefore, this condition appears to be of 
independent interest. 

\begin{definition}
[\protect{for example, \cite[\S 7.3]{aljadeff-belov} or \cite{belov-rowen}}]
A finite subset $S$ of an algebra $A$ is called a 
\emph{Shirshov base of $A$ of height $h$} if $A$ is linearly spanned
by elements of the form $a_1^{k_1} \dots a_n^{k_n}$, where $n \le h$, 
$k_1, \dots, k_n \in \mathbb N$, and $a_i\in S$.
\end{definition}

Thus Shirshov's height theorem says that a finitely-generated PI algebra
has a Shirshov base. The converse is, obviously, not true:
if $L$ is any nonabelian fi\-ni\-te-di\-men\-si\-o\-nal Lie algebra
defined over a field of characteristic zero, then its universal enveloping 
algebra has a Shirshov base of height $\dim L$ due to the 
Poincar\'e-Birkhoff-Witt theorem, and is not PI due to a result of Bahturin
(see \cite[\S 6.7.1]{bahturin-book}).
One can easily construct a multitude of other, ad-hoc, examples of non-PI 
algebras having a Shirshov base: one of the simplest such examples is an algebra
generated by two elements $x$, $y$ subject the single relation $xy = y^2 x$. 

\bigskip

Now let us turn to another circle of results in PI theory, establishing
PIness of an algebra from PIness of its subalgebra distinguished with the help 
of an additional structure (such as being a set of fixed points with respect to 
some action). In particular, consider an algebra graded by a group $G$:
\begin{equation}\label{grading}
A = \bigoplus_{g\in G} A_g \>,
\end{equation}
where $A_g A_h \subseteq A_{gh}$ for any $g,h\in G$.
The neutral component $A_e$, where $e$ is the neutral element of $G$, 
forms a subalgebra whose properties determine,
to a certain extent, the properties of the whole $A$.

\begin{theorem}[PIness of a graded algebra]\label{theorem-gr}
Let $A$ be an algebra graded by a finite group. If $A_e$ is PI, then $A$ is PI.
\end{theorem}

This theorem was first proved by Bergen and Cohen, and elucidated further
(estimates on the degree of polynomial identities, relaxation of initial assumptions) by 
Bahturin, Giambruno and Riley. 
See the book \cite[Theorem 4.3.3]{gz} and the survey \cite[\S 2]{bahturin} 
for details. 

The purpose of this note is to prove a similar theorem, where PIness is replaced
by the property to have a Shirshov base.

\begin{theorem}[Shirshov base of a graded algebra]\label{main}
Let $A$ be a finitely-generated algebra graded by a finite group. 
If $A_e$ has a Shirshov base of height $h$, then $A$ has a Shirshov base of 
height $<(h+1)|G|$.
\end{theorem}

This theorem is related to Proposition 7.10 from \cite{aljadeff-belov}
which states that a Shirshov base of a graded finitely-generated PI algebra $A$ 
can be chosen ``mainly'' from elements of $A_e$.

One (``combinatorial'') variant of the proof of Theorem \ref{theorem-gr}
involves a certain technical lemma about choosing in an arbitrary 
finite sequence of elements of a finite group, consecutive subsequences such that
the product of all elements in each subsequence is equal to $e$ (see Lemma 2.2.2
in the survey \cite{bahturin}).
Our proof of Theorem \ref{main} is of similar character, utilizing a similar 
simple lemma, but this time we are concerned not about subsequences being 
consecutive, but that their total length would be ``big enough''.

\begin{lemma}\label{petrov}
Let $G$ be a finite group, and $g_1, \dots, g_n \in G$. Then there is
a set of non-overlapping intervals consisting of consecutive elements of the sequence
$\{g_1, \dots, g_n\}$ such that the product of all elements inside each interval is equal 
to $e$, and the sum of lengths of all intervals is greater than $n - |G|$.
\end{lemma}

\begin{proof}
Let $\mathcal S$ be a set of non-overlapping intervals such that the product of elements inside
each interval is equal to $e$, and $\mathcal S$ has a maximal possible sum of lengths of 
intervals among 
all sets of intervals with such properties. Let 
$T = [1,n] \,\backslash\, (\, \bigcup_{I\in \mathcal S} I \,)$.

Define a map $f: [1,n] \to G$ by the rule $f(k) = g_1 g_2 \dots g_k$.
Suppose that $f(k) = f(\ell)$ for some $k,\ell \in T$, $k < \ell$. 
Then $g_{k+1} g_{k+2} \dots g_\ell = e$.
If no interval in $\mathcal S$ intersects the interval $[k+1,\ell]$, then adding 
the latter interval to $\mathcal S$, we get a contradiction with assumption of maximality
of $\mathcal S$.
If some interval $I\in \mathcal S$ intersects $[k+1,\ell]$, then, 
since $k, \ell \notin I$, $[k+1,\ell]$ strictly contains $I$.
Then, replacing $I$ by $[k+1,\ell]$, we again get a contradiction with the
maximality of $\mathcal S$. 

This shows that $f$, being restricted to $T$, is an injection.
Consequently, $|T| \le |G|$, and $\sum_{I\in \mathcal S} |I| > n - |G|$.
\end{proof}

This lemma can be generalized by considering rectangles instead of intervals,
and a particular case of such generalization, formulated in a slightly more 
entertaining way for the case of $G = \mathbb Z/17\mathbb Z$ and the rectangle
$100 \times 100$, was presented as a problem at the 
XXXVI Russian Mathematical Olympiad (see \cite{kvant}).

\begin{proof}[Proof of Theorem \ref{main}]
Let $X$ be a finite generating set of an algebra $A$ with grading 
(\ref{grading}). Decomposing elements of $X$ into homogeneous components, we may
assume that $X$ consists of homogeneous elements: $X = \bigcup_{g\in G} X_g$.
Obviously, it is also enough to consider only homogeneous elements of $A$. Let 
$x_{g_1} x_{g_2} \dots x_{g_n}$, where $x_g \in X_g$, be an arbitrary such 
element. By Lemma, we may represent it as the word
\begin{equation}\label{expr}
y_1 a_1 y_2 a_2 \dots y_k a_k ,
\end{equation}
where $a_i$'s are products of $x_g$'s
from the claimed intervals (and thus each $a_i \in A_e$),
and $y_i$'s are products of $x_g$'s from the remaining intervals.
It could be that the first $y_1 = e$ (i.e., the word actually starts with $a_1$),
depending on whether or not the first interval provided by Lemma starts at the 
point $1$, and, similarly, the last $a_k = e$ (i.e., the word ends with $y_k$),
depending on whether or not the last interval provided by Lemma ends at the 
point $n$.

For a word $w$ in $x_g$'s, denote by $\Length(w)$ its length.
We have 
$$
\sum_{i=1}^k \Length(a_i) \ge n-|G|+1 ,
$$
and hence
$$
\sum_{i=1}^k \Length(y_i) = n - \sum_{i=1}^k \Length(a_i) \le |G|-1 .
$$ 
In particular, the number of all nontrivial $y_i$'s, equal to $k$ or $k-1$, is 
$\le |G|-1$.

Let $S$ be a Shirshov base of $A_e$ of height $h$. Then, 
in expression (\ref{expr}) each $a_i$ is represented as a product of $\le h$
powers of elements of $S$, and each $y_i$ is represented as a product of 
$\Length(y_i)$ elements of $X$. Hence (\ref{expr}) is represented as a product
of 
$$
\le hk + \sum_{i=1}^k \Length(y_i) \le (h+1)|G| - 1
$$
powers of elements from $S \cup X$.
\end{proof}

This proof suggests that the given estimate on the height of a Shirshov base 
of $A$ is very rough and probably could be improved by a bit more sophisticated 
combinatorial arguments.

\bigskip

Whether Theorem \ref{main} can be extended to other classes of (nonassociative) 
algebras? It seems likely so, but it will require new ideas, the present proof
does not generalize to the nonassociative case. Note in this connection that 
Shirshov's height theorem is not true in nonassociative setting in general, 
though some variants of it have been established
for certain class of Lie algebras by Michshenko,
for alternative and $(-1,1)$-algebras by Pchelintsev, 
and for alternative and Jordan algebras by Belov 
(see \cite{bbl}, especially Theorem 2.165 and \S 2.2.4, and 
\cite[\S 2]{belov-rowen}).
The significance of these nonassociative analogs, though, seems to be less
then in the associative case. An analog of Theorem \ref{theorem-gr}
has been established for the so-called Lie type algebras (roughly, algebras
satisfying the condition $a(bc) = \alpha (ab)c + \beta(ac)b$ for any elements
$a,b,c$ of an algebra and some base field elements $\alpha$, $\beta$ depending, 
generally, on $a,b,c$, or a graded version thereof) by Bahturin in Zaicev
(see \cite[\S 3.2]{bahturin}).

\section*{Acknowledgements}

Thanks are due to Alexei Belov(-Kanel) and Irina Sviridova 
for useful remarks.
Petrov was supported
by the Russian Government project 11.G34.31.0053
and RFFI grants 11-01-12092-ofi-m-2011 and 11-01-00677-a.
Zusmanovich was supported by grants ERMOS7 
(Estonian Science Foundation and Marie Curie Actions) and ETF9038 
(Estonian Science Foundation).

\end{document}